\documentclass[11pt]{article}
by-\headheight \advance \topmargin by-\headsep 
\usepackage{latexsym}
\usepackage{amsmath}
\usepackage{amssymb}
\usepackage{amsthm}
\usepackage{color}
\usepackage{colortbl}
\allowdisplaybreaks
\usepackage{amsmath,amsfonts,latexsym,amsthm,amsxtra,amssymb,bbm,cases}
\usepackage{mathrsfs}
\usepackage{amsfonts}
\usepackage{amsmath,epsfig,multicol}
\usepackage{amsfonts,  latexsym,  amsthm,  amsxtra,  amssymb, bbm}
\usepackage[numbers,sort&compress]{natbib}
\newtheorem{thm}{Theorem}[section]

\newtheorem{lemma}[thm]{Lemma}

\newtheorem{corollary}[thm]{Corollary}

\newtheorem{proposition}[thm]{Proposition}

\theoremstyle{definition}

\newtheorem{remarkn}[thm]{Remark}

\numberwithin{equation}{section}

\newcommand{\R}{\mathbb{R}}

\newcommand{\N}{\mathbb{N}}

\newcommand{\biindice}[3]%
{
\renewcommand{\arraystretch}{0.5}
\begin{array}[t]{c}
#1\\
{\scriptstyle #2}\\
{\scriptstyle #3}\\
\end{array}
\renewcommand{\arraystretch}{1}
}

\def\biindice#1#2#3{#1_{\textstyle{#2\atop #3}}}

\newcommand{\ds}{\displaystyle}

\begin{document}
\title{\bf Existence and multiplicity results  for the zero mass Schr\"{o}dinger-Bopp-Podolsky system  with critical growth
\thanks{The first author is supported by the National Natural Science Foundation of China (12001198)
and the Natural Science Foundation of Jiangxi Province (20232BAB201009). The second author is
supported by the National Natural Science Foundation of China (12161038).
}}
\author{Wentao Huang
\thanks {Corresponding author. School   of  Science, East China JiaoTong University, Nanchang, 330013, China  (wthuang1014@aliyun.com).} \ \
 and Li Wang
\thanks {School   of  Science, East China JiaoTong University, Nanchang, 330013, China  (wangli.423@163.com).}
}

\date{}
\maketitle

\begin{abstract}
In this paper we study the  following zero mass  Schr\"{o}dinger-Bopp-Podolsky system with critical growth
\begin{equation*}
\renewcommand{\arraystretch}{1.25}
\begin{array}{ll}
\ds \left \{ \begin{array}{ll}
\ds-\Delta  u+q^2\phi u=\mu|u|^{p-2}u+|u|^{4}u,\quad &x\in \R^3, \\
-\Delta\phi+a^2\Delta^2\phi=4\pi u^2,\quad& x\in \R^3,\\
\end{array}
\right.
\end{array}
\end{equation*}
where $a>0$, $q\neq0$, $\mu>0$ is a parameter and $p\in(3,6)$.
By introducing a new functional framework developed by Caponio et al. \cite{Cd}, we first establish the existence of positive ground state solutions
for the case of $p\in(3,6)$. Moreover, for the case of  $p\in(4,6)$, multiplicity results are obtained by applying an abstract critical point theorem due to Perera \cite{Pe}.
\end{abstract}



\section{\label{Int}Introduction and main results}

Recently, there have been growing interests in the study of the following Schr\"{o}dinger-Bopp-Podolsky system
\begin{equation}\label{eqS1.1}
\renewcommand{\arraystretch}{1.25}
\begin{array}{ll}
\ds \left \{ \begin{array}{ll}
\ds-\Delta  u+\omega u+q^2\phi u=|u|^{p-2}u,\quad &x\in \R^3, \\
-\Delta\phi+a^2\Delta^2\phi=4\pi u^2,\quad& x\in \R^3,\\
\end{array}
\right.
\end{array}
\end{equation}
where $u,\ \phi:\R^3\rightarrow\R$, $\omega,\ a>0$ and $q\neq0$. In \cite{d}, system \eqref{eqS1.1} arises when  the authors coupled
a Schr\"{o}dinger field $\psi = \psi(t, x)$ with its electromagnetic field in the Bopp-Podolsky theory, specifically
 in the electrostatic case for the standing waves $\psi(t, x) = e^{i\omega t}u(x)$. In \eqref{eqS1.1}, $q\neq0$ denotes the coupling constant of the interaction between the particle
and its electromagnetic field. $\omega$ is the frequency of the standing wave $\psi(t, x) = e^{i\omega t}u(x)$.
The unknown $u$ represents the modulus of the wave function of the particle and $\phi$ is the
electrostatic potential.

The Bopp-Podolsky theory developed independently by Bopp \cite{Bo} and Podolsky \cite{Po} is a second-order
gauge theory for the electromagnetic field.
It was also introduced to address  the so-called infinity problem associated with
a point charge in the classical Maxwell theory, see \cite{B1,B2,B3,B4,Mi}. In fact,  in this Maxwell theory, by the well-known Gauss law
(or Poisson equation), for a given charge distribution with density $\rho$, the electrostatic potential $\phi$  satisfies
the equation
\begin{equation}\label{eqS1.2}
-\Delta \phi=\rho,\quad x\in \R^3.
\end{equation}
Let $\rho=4\pi \delta_{x_0}$ with $x_0\in\R^3$, then equation \eqref{eqS1.2} has a fundamental solution $\mathcal{G}(x-x_0)$ with
$$
\mathcal{G}(x)=\frac{1}{|x|},
$$
and  the energy of the corresponding electrostatic field is not finite since
$$
\mathcal{E}_\text{M}=\frac12\int_{\R^3}|\nabla\mathcal{G}|^2dx=+\infty.
$$
However,  in the Bopp-Podolsky theory, equation \eqref{eqS1.2} is replaced by
\begin{equation}\label{eqS1.3}
-\Delta \phi+a^2\Delta^2\phi=\rho,\quad x\in \R^3.
\end{equation}
Take $\rho=4\pi \delta_{x_0}$ again, the fundamental solution of equation \eqref{eqS1.3} is $\mathcal{K}(x-x_0)$ with
$$
\mathcal{K}(x)=\frac{1-e^{-\frac {|x|}a}}{|x|},
$$
and  its energy is  finite since
$$
\mathcal{E}_\text{BP}=\frac12\int_{\R^3}|\nabla\mathcal{K}|^2dx+\frac{a^2}2\int_{\R^3}|\Delta\mathcal{K}|^2dx<+\infty.
$$

In addition, the Bopp-Podolsky theory can be regarded as effective theory for short distances, and it is experimentally
indistinguishable from the Maxwell theory for large distance, see  \cite{Fr}. Here  $a > 0$
is the Bopp-Podolsky parameter that has the dimension of the inverse of mass and can be interpreted as a cut-off distance or can be linked to an effective
radius for the electron.

Motivated by the strong physical background, there have been growing interest around system \eqref{eqS1.1} in the last few years.
In \cite{d}, based on variational methods, d'Avenia and Siciliano first studied  system \eqref{eqS1.1} from a mathematical point of view.
Precisely, they proved that \eqref{eqS1.1} admits a nontrivial solution when $p \in (2, 6)$ and $|q|$ small enough or $p \in (3, 6)$ and $q\neq0$,
and also showed that, in the radial case,  the solutions tend to  solutions of the classical Schr\"{o}dinger-Poisson
system as $a\rightarrow0$. Meanwhile, they obtained that
\eqref{eqS1.1} does not admit any nontrivial solution for $p \ge 6$ via Pohozaev identity.
Subsequently, Siciliano and Silva \cite{SS} proved that for $p \in (2, 3]$,
\eqref{eqS1.1} has no solution for large $q$ but admits two radial solutions for  small $q$   via the fibering approach.
They also given qualitative properties about the energy level of the solutions and a variational characterization of these extremal values of $q$.
Using different variational techniques, \cite{CT,LPT} considered system \eqref{eqS1.1} in the critical case and obtained the existence of  ground state solutions.
Furthermore, we refer the readers
to the interesting results obtained in \cite{WCL,Z1,Z2} concerning sign-changing solutions, in \cite{dS,HW,LZ,LCF} for normalized solutions,
in \cite{DS,de} for semiclassical states and in \cite{dG,H1,H2}
 for nonlinear Schr\"{o}dinger-Bopp-Podolsky-Proca system on manifolds.

Note that all the papers mentioned above to system \eqref{eqS1.1} only deal with $\omega>0$. When $\omega=0$, system \eqref{eqS1.1} reduces to the following
static case
\begin{equation}\label{eqS1.4}
\renewcommand{\arraystretch}{1.25}
\begin{array}{ll}
\ds \left \{ \begin{array}{ll}
\ds-\Delta  u+q^2\phi u=|u|^{p-2}u,\quad &x\in \R^3, \\
-\Delta\phi+a^2\Delta^2\phi=4\pi u^2,\quad& x\in \R^3.\\
\end{array}
\right.
\end{array}
\end{equation}
Problem \eqref{eqS1.4} can be thought of as a zero mass problem  (see \cite{BL}), which arises in certain problems
related to the Yang-Mills equations.

Furthermore, let $a=0$, system \eqref{eqS1.4} turns into the zero mass Schr\"{o}dinger-Poisson
system
\begin{equation}\label{eqS1.5}
\renewcommand{\arraystretch}{1.25}
\begin{array}{ll}
\ds \left \{ \begin{array}{ll}
\ds-\Delta  u+q^2\phi u=|u|^{p-2}u,\quad &x\in \R^3, \\
-\Delta\phi=4\pi u^2,\quad& x\in \R^3.\\
\end{array}
\right.
\end{array}
\end{equation}
It is well known that the second equation in system \eqref{eqS1.5} is called the Poisson equation, which can be solved by
$$
\phi(x)=\frac{1}{|x|}*u^2.
$$
In such case, problem \eqref{eqS1.5} can be rewritten as a single equation in the following
form:
\begin{equation}\label{eqS1.6}
-\Delta u+q^2\left(\frac{1}{|x|}*u^2\right)u=|u|^{p-2}u,\quad x\in \R^3.
\end{equation}
The absence of a phase term makes the usual Sobolev space $H^1(\R^3)$ not to
be a good framework for studying the problem \eqref{eqS1.6}. It seems
quite clear that the right space should be (see \cite{R}):
$$
E(\R^3)=\left\{u\in D^{1,2}(\R^3)\ \Big|\ \int_{\R^3} \int_{\R^3} \frac{u^2(x)u^2(y)}{|x-y|}dxdy<+\infty\right\}.
$$
Based on this, for $p\in(\frac{18}{7},3)$, Ruiz \cite{R} showed that the corresponding functional has a minimizer at a negative level, which is a positive solution of \eqref{eqS1.6} in $E(\R^3)$. For $p\in(3,6)$,  Ianni and Ruiz \cite{IR} obtained the existence of a ground state solution by using monotonicity trick and a
compactness splitting lemma. Moreover,  they were concerned with the existence of infinitely many radial bound
state solutions via Krasnoselskii genus. Later, the higher-dimensional version of zero mass Schr\"{o}dinger-Poisson
system \eqref{eqS1.5} has also been investigated in \cite{MMV}. Recently, Liu et al. \cite{LZH} generalized the results in \cite{IR} to the critical case that
replacing $|u|^{p-2}u$ by $\mu|u|^{p-2}u+|u|^{4}u$.
Particularly, by combining a new perturbation approach and the well-known mountain pass theorem, when $p \in (3, 6)$,
they obtained the existence of positive ground state solutions.

Notably, inspired by \cite{R}, the authors in \cite{Cd}
introduced an appropriate working space and established some fundamental properties for this space including
embeddings into Lebesgue spaces. Moreover a general lower bound for the Bopp-Podolsky energy was
obtained. Based on these facts, by applying a perturbation argument, they finally proved the existence of a
weak solution to system \eqref{eqS1.4} for $p\in(3,6)$. Motivated by the cited papers \cite{Cd,LZH}, the main goal of this paper  is to investigate the zero
mass Schr\"{o}dinger-Bopp-Podolsky system with critical growth, namely
\begin{equation}\label{eqS1.7}
\renewcommand{\arraystretch}{1.25}
\begin{array}{ll}
\ds \left \{ \begin{array}{ll}
\ds-\Delta  u+q^2\phi u=\mu|u|^{p-2}u+|u|^{4}u,\quad &x\in \R^3, \\
-\Delta\phi+a^2\Delta^2\phi=4\pi u^2,\quad& x\in \R^3,\\
\end{array}
\right.
\end{array}
\end{equation}
where $\mu>0$ is a parameter and $p\in(3,6)$.

In this paper we mainly investigate two types of results. The first type is concerned with the existence of positive ground state solutions to system \eqref{eqS1.7}
for the case of $p\in(3,6)$.
Instead of applying some perturbation argument as in \cite{Cd,LZH}, we will employ a simpler method to address it.
Our first result read as

\begin{thm}\label{thm1.1}
If $p\in(4,6)$, system \eqref{eqS1.7} has a positive ground state solution for all $\mu>0$, while if $p\in(3,4]$, system \eqref{eqS1.7} still has a positive ground state solution for $\mu>0$ sufficiently large.
\end{thm}

The proof of Theorem \ref{thm1.1} is based on variational method. The main difficulties in proving Theorem \ref{thm1.1} lie in two aspects. Firstly,
it is not yet known if the (PS) sequences are bounded or not, since the corresponding (AR) condition does not hold for $p\in(3,4]$.
In a prior work \cite{LZH}, the authors dealt with
this issue by adding a nonlocal  perturbation term  into the corresponding functional. We also note that this approach is helpful to recover the boundedness of the (PS) sequence. However, different from \cite{LZH}, we shall analyze the natural functional  directly in the current paper,
therefore is somewhat simpler. Precisely,
motivated by \cite{HIT,J},
we first study the behavior of $J(e^{2\theta}v(e^{\theta}x))$ for $\theta\in\R$,
then we construct a $\text{(PS)}_{c}$ sequence $\{u_n\}$ with an extra asymptotically property  (see Proposition \ref{pro3.4} below)
\begin{equation}\label{eqS1.8}
\renewcommand{\arraystretch}{1.25}
\ds G(u_n)\rightarrow0\quad\text{as}\ n\rightarrow\infty,
\end{equation}
where $c$ is the mountain pass level of the functional $J$ and  $G(u)=2\big\langle J'(u),u\big\rangle-P(u)$
(see \eqref{eqS3.2} below).
From \eqref{eqS1.8}, we can obtain the boundedness of the $\text{(PS)}_{c}$  sequence easily.
Secondly, the unboundedness of the domain $\R^3$ and the nonlinearity
with  critical growth lead to the lack of compactness.
This problem can be solved by giving an important upper bounded estimate on the  mountain pass level $c$ together with
compactness lemma due to Lions \cite{Lions}.
Proceeding by the standard variational method, the existence of ground state solutions to problem  \eqref{eqS1.7} follows.

The second type of results deals with the multiplicity result of system \eqref{eqS1.7}
for the case of $p\in(4,6)$. To the best of our knowledge, there is no result in the literature on the existence of multiple solutions to
Schr\"{o}dinger-Bopp-Podolsky system in the zero mass case. Thanks to an abstract critical point theorem recently obtained by Theorem 2.1 in \cite{Pe},
we will prove our second result as follows.

\begin{thm}\label{thm1.2}
Let $p\in(4,6)$. For any $m\in\N$, there exists $\mu_m>0$ such that system \eqref{eqS1.7} has $m$ pairs of solutions with positive energy for all $\mu\ge\mu_m$.
\end{thm}

To get multiplicity results like this theorem, a crucial step is to prove the (PS) condition.
Since system \eqref{eqS1.7} involves critical growth, the most we can expect is the local (PS) condition, that is the  $\text{(PS)}_{c}$ condition
for $c\in(0,c^*)$ being $c^*$ some positive constant.

The outline of this paper is as follows. In Section 2,  some preliminary results and lemmas are presented.
In Section 3, we  show the  existence of positive ground state solutions to system \eqref{eqS1.7}
for the case of $p\in(3,6)$. Section 4 is devoted to the multiplicity result of system \eqref{eqS1.7}
for the case of $p\in(4,6)$.

Finally, we present some notations used in this paper.

\noindent
$\bullet$ \ $L^q(\R^3)$ is the usual Lebesgue space with the standard norm
$$
\|u\|_q:=\Big(\int_{\R^3}|u|^qdx\Big)^\frac1q,\ 1\leq q<\infty.
$$
$\bullet$ \ $X^{-1}$ denote the dual space of $X$.\\
$\bullet$ \ For any $x\in\R^3$ and $R>0$, $B_R(x):=\{y\in\R^3\ \big|\ |y-x|<R\}$.\\
$\bullet$ \ $o_n(1)$ means a vanishing sequence as $n\rightarrow\infty$.\\
$\bullet$ \ $C$ will be used to denote various positive constants which may change from line to line.

\section{Notations and preliminaries}

In this section, we provide some notations and preliminary lemmas. Let
\[\mathcal{A}: = \Big\{ \phi \in D^{1,2}({\R^3})\ \big|\ \Delta\phi \in {L^2}({\R^3})\Big\} \]
equipped with the scalar product
$$
(\phi,\varphi)_{\mathcal{A}}: = \int_{{\R^3}} \nabla\phi\nabla\varphi dx  + a^2\int_{{\R^3}} \Delta\phi\Delta\varphi dx.
$$
We denote by $\|\cdot\|_{\mathcal{A}}$ the associated norm. From \cite{d}, we know that $C_0^\infty(\R^3)$ is dense in the Hilbert space $\mathcal{A}$ and
$$
\mathcal{A}\hookrightarrow L^q(\R^3)\quad \text{for}\ q\in[6,+\infty].
$$

In order to study the zero mass problem \eqref{eqS1.7}, we need first to consider an appropriate working space instead of $H^1(\R^3)$. Inspired by
\cite{Cd}, we introduce the space $E$ of the functions in $D^{1,2}({\R^3})$ with finite Bopp-Podolsky energy, i.e.,
$$
E:=\left\{u\in D^{1,2}(\R^3)\ \Big|\ \int_{\R^3} \int_{\R^3} \frac{1-e^{-\frac{|x-y|}{a}}}{|x-y|}u^2(x)u^2(y)dxdy<+\infty\right\}
$$
equipped with the norm
$$
\|u\|:=\left[\int_{\R^3}|\nabla u|^2dx+\Bigg(\int_{\R^3} \int_{\R^3} \frac{1-e^{-\frac{|x-y|}{a}}}{|x-y|}u^2(x)u^2(y)dxdy\Bigg)^\frac12\right]^{\frac12}.
$$
It has been proved in \cite{Cd} that
$$
E=\Big\{ u \in D^{1,2}({\R^3})\ \big|\ \phi_u \in \mathcal{A}\Big\},
$$
where
$$
\phi_u(x):=(\mathcal{K}*u^2)(x)= \int_{\R^3} \frac{1-e^{-\frac{|x-y|}{a}}}{|x-y|}u^2(y)dy.
$$
Moreover, for each $u\in E$, $\phi_u$ is the unique weak solution in $\mathcal{A}$ of
$$
-\Delta\phi+a^2\Delta^2\phi=4\pi u^2,\quad x\in \R^3,
$$
and
\begin{equation}\label{eqS2.1}
\ds \|\phi_u\|_\mathcal{A}^2=\int_{\R^3}\left(|\nabla \phi_u|^2+a^2|\Delta\phi_u|^2\right)dx=4\pi\int_{\R^3}\phi_uu^2dx=4\pi\int_{\R^3} \int_{\R^3} \frac{1-e^{-\frac{|x-y|}{a}}}{|x-y|}u^2(x)u^2(y)dxdy.
\end{equation}

In addition, the functional space $E$ has the following useful embedding properties.
\begin{lemma}\label{lm2.1}(\cite{Cd})
The space $(E, \|\cdot\|)$ is a uniformly convex Banach space. Moreover, the space $C_0^\infty(\R^3)$ is dense in $(E, \|\cdot\|)$. The space $E$
is continuously embedded into $L^q(\R^3)$ for $q\in[3,6]$.
\end{lemma}

Moreover, denoting by $E_r$ the subspace of radial functions in $E$, we can say something more on the embedding properties.
\begin{lemma}\label{lm2.2}(\cite{Cd})
The space $E_r$
is continuously embedded into $L^q(\R^3)$ for $q\in(\frac{18}{7},6]$. The embedding is compact for $q\in(\frac{18}{7},6)$.
\end{lemma}

For $\phi_u$, we have the following convergence properties.
\begin{lemma}\label{lm2.3}(\cite{Cd})
Let $\{u_n\}$ be a sequence in  $E$. We have that

\noindent
(i) \ if $u_n\rightharpoonup u$ in $E$, then $\phi_{u_n}\rightharpoonup\phi_u$ in $\mathcal{A}$;\\
(ii) \ $u_n\rightarrow u$ if and only if $u_n\rightarrow u$ in $D^{1,2}(\R^3)$ and $\phi_{u_n}\rightarrow\phi_u$ in $\mathcal{A}$.
\end{lemma}

We also state here an adaptation to the space $E$ of a result due to Lions, see Lemma I.1 of \cite{Lions}.
\begin{lemma}\label{lm2.4}
Let $\{u_n\}$ be a bounded sequence in  $E$, $q\in[3,6)$, and assume that
$$
\sup\limits_{y\in\R^3}\int_{B_R(y)}|u_n|^qdx\rightarrow0\quad\text{for some}\ R>0.
$$
Then $u_n\rightarrow0$ in $L^\alpha(\R^3)$ for any $\alpha\in(3,6)$.
\end{lemma}

\begin{proof}
By applying Lemma I.1 of \cite{Lions} with $p=2$, we obtain that $u_n\rightarrow0$ in $L^\alpha(\R^3)$ for any $\alpha\in(q,6)$.
Recall now that  $\{u_n\}$ is bounded in  $E$, and hence in $L^3(\R^3)$ by Lemma \ref{lm2.1}. We conclude by interpolation.
\end{proof}

The following lemma is a generalization of Proposition 2.2 in \cite{SS} for $u\in H^1(\R^3)$.
\begin{lemma}\label{lm2.5}(\cite{Cd})
For all $u\in E$,
$$
\|u\|_3^3\le\frac1\pi\|\phi_u\|_\mathcal{A}\|\nabla u\|_2.
$$
\end{lemma}

Define $M:E\rightarrow\R$ as
$$
M(u):=\int_{\R^3}|\nabla u|^2dx+\int_{\R^3}\phi_u u^2dx.
$$
We can easily obtain the fact that for any $u\in E$,
\begin{equation}\label{eqS2.2}
\frac12\|u\|^4\le M(u)\le\|u\|^2,\quad\text{if either}\ \|u\|\le1\ \text{or}\ M(u)\le1.
\end{equation}
At this point, we have
\begin{lemma}\label{lm2.6}
There exists $C>0$ such that
$$
\|u\|_p^p\le CM(u)^{\frac{2p-3}{3}}
$$
for $u\in E$ with $p\in[3,6]$.
\end{lemma}

\begin{proof}
Since $D^{1,2}({\R^3})\hookrightarrow L^6(\R^3)$, for $\lambda=\frac{6-p}{p}\in[0,1]$ and by Lemma \ref{lm2.5}, one has
\begin{equation*}
\begin{split}
\ds \|u\|_p&\le\|u\|_3^\lambda\|u\|_6^{1-\lambda}\\[1mm]
  &\le C\left(\|\phi_u\|_\mathcal{A}\|\nabla u\|_2\right)^\frac\lambda3\|\nabla u\|_2^{1-\lambda}\\[1mm]
  &\le C\left(\|\phi_u\|_\mathcal{A}^2+\|\nabla u\|_2^2\right)^\frac\lambda3 M(u)^{\frac{1-\lambda}{2}}.
\end{split}
\end{equation*}
From \eqref{eqS2.1}, we get that
$$
\|u\|_p\le CM(u)^\frac\lambda3 M(u)^{\frac{1-\lambda}{2}}=C M(u)^{\frac{3-\lambda}{6}}.
$$
So the conclusion follows.
\end{proof}

The following General Minimax Principle is due to  Willem \cite{W}.

\begin{lemma}\label{lm2.7}(\cite{W})
Let $X$ be a Banach space. Let $M_0$ be a closed subspace of the metric space $M$ and $\Gamma_0\subset C(M_0,X)$. Define
$$
\Gamma:=\big\{\eta\in C(M,X)\ \big|\ \eta\big|_{M_0}\in\Gamma_0\big\}.
$$
If $\varphi\in C^1(X,\R)$ satisfies
$$
\infty>c:=\inf\limits_{\eta\in\Gamma}\sup\limits_{u\in M}\varphi(\eta(u))>a:=\sup\limits_{\eta_0\in\Gamma_0}\sup\limits_{u\in M_0}\varphi(\eta_0(u)),
$$
then, for every $\varepsilon\in(0,(c-a)/2),\ \delta>0$ and $\eta\in\Gamma$ such that $\sup\limits_{M}\varphi\circ\eta\leq c+\varepsilon$,
there exists $u\in X$ such that

\noindent
(a) \ $c-2\varepsilon\leq\varphi(u)\leq c+2\varepsilon$,\\
(b) \  $dist(u,\eta(M))\leq2\delta$,\\
(c) \  $\|\varphi'(u)\|\leq8\varepsilon/\delta$.
\end{lemma}

\section{\label{sec3}Existence for $p\in(3,6)$}

In this section, we study the existence of  ground state solutions to system \eqref{eqS1.7} in the case $p\in(3,6)$.
The energy functional  associated with system \eqref{eqS1.7} is denoted as
$$
J(u)=\frac12\int_{\R^3}|\nabla u|^2dx+\frac{q^2}{4}\int_{\R^3}\phi_u u^2dx-\frac{\mu}{p}\int_{\R^3}|u|^pdx-\frac{1}{6}\int_{\R^3}|u|^6dx.
$$
In view of \cite{Cd}, $J$ is a $C^1$ functional on $E$ and
$$
\big\langle J'(u),\varphi\big\rangle=\int_{\R^3}\nabla u\nabla\varphi dx+q^2\int_{\R^3}\phi_uu\varphi dx-\mu\int_{\R^3}|u|^{p-2}u\varphi dx
-\int_{\R^3}|u|^{4}u\varphi dx,
$$
for all $\varphi\in E$. Moreover, critical points of $J$ in $E$ corresponding to weak solutions of system \eqref{eqS1.7}.
If $u\in E$ is a weak solution of system \eqref{eqS1.7},
by  \cite{d} we can deduce the following Pohozaev type identity
\begin{equation}\label{eqS3.1}
\begin{split}
\ds P(u)&:=\frac{1}{2} \int_{\R^3}|\nabla u|^2dx+\frac{q^2}{4}\int_{\R^3} \int_{\R^3} \left[\frac{5\left(1-e^{-\frac{|x-y|}{a}}\right)}{|x-y|}+\frac{e^{-\frac{|x-y|}{a}}}{a}\right]u^2(x)u^2(y)dxdy\\[1mm]
  &\quad -\frac{3\mu }{p} \int_{\R^3}|u|^pdx-\frac{1}{2}\int_{\R^3}|u|^{6}dx=0.
\end{split}
\end{equation}
Now,
we introduce the following Nehari-Pohozaev manifold:
$$
\mathcal{M}:=\big\{u\in E\setminus\{0\}\ \big|\ G(u)=0\big\},
$$
where
\begin{equation}\label{eqS3.2}
\begin{split}
\ds G(u)&:=\frac{3}{2}\int_{\R^3}|\nabla u|^2dx+\frac{3q^2}{4}\int_{\R^3} \int_{\R^3} \frac{1-e^{-\frac{|x-y|}{a}}-\frac{|x-y|}{3a}e^{-\frac{|x-y|}{a}}}{|x-y|}u^2(x)u^2(y)dxdy\\[1mm]
  &\quad -\frac{\mu(2p-3)}{p}  \int_{\R^3}|u|^pdx-\frac{3}{2}\int_{\R^3}|u|^{6}dx\\[1mm]
  &=2\big\langle J'(u),u\big\rangle-P(u).
\end{split}
\end{equation}

\begin{remarkn}\label{re3.1}
If $u\in E$ is a nontrivial weak solution of system \eqref{eqS1.7},  by \eqref{eqS3.1} and \eqref{eqS3.2} we see that $u\in\mathcal{M}$.
\end{remarkn}

\begin{lemma}\label{lm3.2}
For any $u\in E\setminus\{0\}$, there is a unique $\tilde{t}>0$ such that $u_{\tilde{t}}\in\mathcal{M}$, where
$u_{\tilde{t}}(x)=\tilde{t}^2u(\tilde{t}x)$. Moreover, $J(u_{\tilde{t}})=\max\limits_{t>0}J(u_t)$.
\end{lemma}

\begin{proof}
For any $u\in E\setminus\{0\}$ and $t>0$, set $u_{t}(x):=t^2u(tx)$. Consider
\begin{equation*}
\begin{split}
\ds y(t):= J(u_t)&=
       \frac{1}{2} t^3\int_{\R^3}|\nabla u|^2dx+\frac{q^2}{4}t^3\int_{\R^3} \int_{\R^3} \frac{1-e^{-\frac{|x-y|}{ta}}}{|x-y|}u^2(x)u^2(y)dxdy\\[1mm]
  &\quad -\frac{\mu }{p} t^{2p-3}\int_{\R^3}|u|^pdx-\frac{1}{6}t^9\int_{\R^3}|u|^{6}dx.
\end{split}
\end{equation*}
It is easy to check that $y(t)>0$ for $t>0$ small and $y(t)\rightarrow-\infty$ as $t\rightarrow+\infty$,
which gives that
 $y(t)$ has a
critical point $\tilde{t}>0$ corresponding to its maximum, i.e., $y(\tilde{t})=\max\limits_{t>0}y(t)$ and $y'(\tilde{t})=0$.
Then
\begin{equation*}
\begin{split}
&\frac{3}{2} \tilde{t}^2\int_{\R^3}|\nabla u|^2dx+\frac{3q^2}{4}\tilde{t}^2\int_{\R^3} \int_{\R^3} \frac{1-e^{-\frac{|x-y|}{\tilde{t}a}}}{|x-y|}u^2(x)u^2(y)dxdy
-\frac{q^2}{4a}\tilde{t}\int_{\R^3} \int_{\R^3} e^{-\frac{|x-y|}{\tilde{t}a}}u^2(x)u^2(y)dxdy\\[1mm]
  &\quad -\frac{\mu (2p-3)}{p} \tilde{t}^{2p-4}\int_{\R^3}|u|^pdx-\frac{3}{2}\tilde{t}^8\int_{\R^3}|u|^{6}dx=0,
\end{split}
\end{equation*}
and hence $G(u_{\tilde{t}})=0,\ u_{\tilde{t}}\in\mathcal{M}$ and $J(u_{\tilde{t}})=\max\limits_{t>0}J(u_t)$.

Next we claim that the critical point of $y(t)$ is unique. Remark that $y'(t)=0$ and $t>0$ imply that
\begin{equation}\label{eqS3.3}
\begin{split}
\frac{3}{2}\int_{\R^3}|\nabla u|^2dx&=\frac{q^2}{4a}t^{-1}\int_{\R^3} \int_{\R^3} e^{-\frac{|x-y|}{ta}}u^2(x)u^2(y)dxdy
-\frac{3q^2}{4}\int_{\R^3} \int_{\R^3} \frac{1-e^{-\frac{|x-y|}{ta}}}{|x-y|}u^2(x)u^2(y)dxdy\\[1mm]
  &\quad +\frac{\mu (2p-3)}{p} t^{2p-6}\int_{\R^3}|u|^pdx+\frac{3}{2}t^6\int_{\R^3}|u|^{6}dx\\[1mm]
  &=:z(t).
\end{split}
\end{equation}
Since $p>3$, a simple calculation shows that $z(t)<0$ for $t>0$ small,  $z(t)\rightarrow+\infty$ as $t\rightarrow+\infty$
and for all $t>0$
\begin{equation*}
\begin{split}
z'(t)&=\frac{q^2}{2a}t^{-2}\int_{\R^3} \int_{\R^3} e^{-\frac{|x-y|}{ta}}u^2(x)u^2(y)dxdy
+\frac{q^2}{4a^2}t^{-3}\int_{\R^3} \int_{\R^3} |x-y|e^{-\frac{|x-y|}{ta}}u^2(x)u^2(y)dxdy\\[1mm]
  &\quad +\frac{\mu (2p-3)(2p-6)}{p} t^{2p-7}\int_{\R^3}|u|^pdx+9t^5\int_{\R^3}|u|^{6}dx>0.
\end{split}
\end{equation*}
Consequently, there exists a unique $\tilde{t}$ such that \eqref{eqS3.3} holds. Hence, the claim follows.
\end{proof}

\begin{lemma}\label{lm3.3}
The functional $J$ satisfies the following:

\noindent
(i) \ there exist $\alpha,\ \rho>0$ such that $J(u)\geq \alpha$ for all $\|u\|=\rho$;\\
(ii) \  there exists  $w\in E$ such that $\|w\|>\rho$ and $J(w)<0$.
\end{lemma}

\begin{proof}
It follows from Lemma \ref{lm2.6} that
\begin{align*}
\begin{split}
\ds J(u)&=\frac 12\int_{\R^3}|\nabla u|^2dx+\frac{q^2}{4}\int_{\R^3}\phi_u u^2dx-\frac{\mu}{p}\int_{\R^3}|u|^pdx-\frac{1}{6}\int_{\R^3}|u|^6dx\\[1mm]
        &\geq \min\left\{\frac12,\frac{q^2}{4}\right\}M(u)-\frac{\mu C}{p}M(u)^{\frac{2p-3}{3}}-\frac{C}{6}M(u)^3\\[1mm]
        &\ge \frac12\min\left\{\frac12,\frac{q^2}{4}\right\}M(u)
\end{split}
\end{align*}
for $M(u)\le\delta$ by choosing $\delta$ small enough. Thanks to \eqref{eqS2.2}, we can take $\delta\in(0,1)$ such that
$$
J(u)\ge\frac12\min\left\{\frac12,\frac{q^2}{4}\right\}M(u)\ge\frac14\min\left\{\frac12,\frac{q^2}{4}\right\}\|u\|^4:=\alpha.
$$
Therefore, we conclude (i).

Fix $u\in E\setminus\{0\}$ and define, for any $t>0$,  $u_t(x)=t^2u(tx)$. We compute
\begin{equation*}
\begin{split}
\ds  J(u_t)
        &=\frac 12 t^3\int_{\R^3}|\nabla u|^2dx+\frac{q^2}{4}t^3\int_{\R^3} \int_{\R^3} \frac{1-e^{-\frac{|x-y|}{ta}}}{|x-y|}u^2(x)u^2(y)dxdy\\[1mm]
  &\quad -\frac{\mu }{p} t^{2p-3}\int_{\R^3}|u|^pdx-\frac{1}{6}t^9\int_{\R^3}|u|^{6}dx\\[1mm]
     &\rightarrow-\infty\quad\text{as}\ t\rightarrow+\infty,
\end{split}
\end{equation*}
 which gives (ii) if we take $w=u_t$ with $t$ sufficiently large.
\end{proof}

In view of Lemma \ref{lm3.3}, we can define the mountain pass level of $J$ as follows:
\begin{equation}\label{eqS3.4}
  c:=\inf\limits_{\eta\in\Gamma}\max\limits_{t\in[0,1]}J(\eta(t))>0,
\end{equation}
where
$$
\Gamma=\big\{\eta\in C([0,1],E)\ \big|\ \eta(0)=0,\ \eta(1)=w\big\}.
$$

Now, we will construct a (PS) sequence $\{u_n\}\subset E$ for $J$ at the level $c$ with
$G(u_n)\rightarrow0$ as $n\rightarrow\infty$.

\begin{proposition}\label{pro3.4}
There exists a sequence $\{u_n\}\subset E$  such that, as $n\rightarrow\infty$
$$
J(u_n)\rightarrow c,\ J'(u_n)\rightarrow0\quad\text{and}\quad G(u_n)\rightarrow0.
$$
\end{proposition}

\begin{proof}
Following the idea in \cite{HIT,J}, we define the map $\Phi:\R\times E\rightarrow E$ for
$\theta\in\R,\ v\in E$ and $x\in\R^3$ by $\Phi(\theta,v)=e^{2\theta}v(e^{\theta}x)$. For every $\theta\in\R,\ v\in E$,
the functional $J\circ\Phi$ is computed as
\begin{equation*}
\begin{split}
\ds   J\circ\Phi(\theta,v)
        &=\frac 12 e^{3\theta}\int_{\R^3}|\nabla v|^2dx+\frac{q^2}{4}e^{3\theta}\int_{\R^3} \int_{\R^3} \frac{1-e^{-\frac{|x-y|}{ae^\theta }}}{|x-y|}v^2(x)v^2(y)dxdy\\[1mm]
  &\quad -\frac{\mu }{p} e^{(2p-3)\theta}\int_{\R^3}|v|^pdx-\frac{1}{6}e^{9\theta}\int_{\R^3}|v|^{6}dx.
\end{split}
\end{equation*}
We can easily check that $J\circ\Phi(\theta,v)>0$ for all $(\theta,v)$ with $|\theta|,\ \|v\|$
small and $J\circ\Phi(0,w)<0$, where $w$ is given in Lemma \ref{lm3.3}. Hence, $J\circ\Phi$ possesses the mountain pass geometry in
$\R\times E$. As a result, we can define the mountain pass  level of $J\circ\Phi$ as follows:
\begin{equation}\label{eqS3.5}
  \tilde{c}:=\inf\limits_{\tilde{\eta}\in\tilde{\Gamma}}\max\limits_{t\in[0,1]}J\circ\Phi(\tilde{\eta}(t)),
\end{equation}
where
$$
\widetilde{\Gamma}=\big\{\tilde{\eta}\in C([0,1],\R\times E)\ \big|\ \tilde{\eta}(0)=(0,0),\ \tilde{\eta}(1)=(0,w)\big\}.
$$

\noindent
Since $\Gamma=\{\Phi\circ\tilde{\eta}\ \big|\ \tilde{\eta}\in\widetilde{\Gamma}\}$, the mountain pass level of $J$
coincides with the mountain pass level of
$J\circ\Phi$, i.e.,
$
c=\tilde{c}.
$
By applying Lemma \ref{lm2.7} with $M=[0,1]$, $M_0=\{0,1\}$, $\Gamma_0=\{\eta_0\}$, $\eta_0(0)=(0,0)$ and $\eta_0(1)=(0,w)$, we see that there exists a sequence $\{(\theta_n,v_n)\}_{n\in\N}$ in $\R\times E$ such that, as $n\rightarrow\infty$
\begin{equation}\label{eqS3.6}
  J\circ\Phi(\theta_n,v_n)\rightarrow c,
\end{equation}
\begin{equation}\label{eqS3.7}
    (J\circ\Phi)'(\theta_n,v_n)\rightarrow 0\quad\text{in}\quad(\R\times E)^{-1},
\end{equation}
\begin{equation}\label{eqS3.8}
  \theta_n\rightarrow0.
\end{equation}

\noindent
Indeed, set $\varepsilon=\varepsilon_n=\frac{1}{n^2},\ \delta=\delta_n=\frac{1}{n}$ in Lemma \ref{lm2.7}, then \eqref{eqS3.6} and \eqref{eqS3.7} are direct
conclusions from $(a)$ and $(c)$ of Lemma \ref{lm2.7}. Next, we just need to verify \eqref{eqS3.8}. In view of the definition of $c$ in \eqref{eqS3.4},
for $\varepsilon=\varepsilon_n=\frac{1}{n^2},\ \exists\eta_n\in\Gamma$ such that
$$
\max\limits_{t\in[0,1]}J(\eta_n(t))\leq c+\frac{1}{n^2}.
$$
Set $\tilde{\eta}_n(t)=(0,\eta_n(t))$, then
$$
\max\limits_{t\in[0,1]}J\circ\Phi(\tilde{\eta}_n(t))=\max\limits_{t\in[0,1]}J(\eta_n(t))\leq c+\frac{1}{n^2}.
$$
By $(b)$ of Lemma \ref{lm2.7}, there exists $(\theta_n,v_n)\in \R\times E$ such that $dist((\theta_n,v_n),(0,\eta_n(t)))\leq\frac2n$,
so \eqref{eqS3.8} holds.

Set $u_n=\Phi(\theta_n,v_n)$,  \eqref{eqS3.6} gives that
\begin{equation*}
  J(u_n)\rightarrow c\quad\text{as}\quad n\rightarrow\infty.
\end{equation*}
Since for any $(h,z)\in\R\times E$,
\begin{equation}\label{eqS3.9}
\begin{split}
 \big\langle(J\circ\Phi)'(\theta_n,v_n),(h,z)\big\rangle&=e^{3\theta_n}\int_{\R^3}\nabla v_n\nabla zdx+q^2e^{3\theta_n}\int_{\R^3} \int_{\R^3}
 \frac{1-e^{-\frac{|x-y|}{ae^{\theta_n }}}}{|x-y|}v_n^2(y)v_n(x)z(x)dxdy\\[1mm]
 &\quad -\mu e^{(2p-3)\theta_n}\int_{\R^3}|v_n|^{p-2}v_nzdx-e^{9\theta_n}\int_{\R^3}|v_n|^4v_nzdx\\[1mm]
 &\quad+\Bigg[\frac{3}{2} e^{3\theta_n}\int_{\R^3}|\nabla v_n|^2dx+\frac{3q^2}{4}e^{3\theta_n}\int_{\R^3} \int_{\R^3} \frac{1-e^{-\frac{|x-y|}{ae^{\theta_n}}}}{|x-y|}v_n^2(x)v_n^2(y)dxdy\\[1mm]
 &\quad\quad\quad-\frac{q^2}{4a}e^{2\theta_n}\int_{\R^3} \int_{\R^3} e^{-\frac{|x-y|}{ae^{\theta_n}}}v_n^2(x)v_n^2(y)dxdy\\[1mm]
 &\quad\quad\quad-\frac{\mu (2p-3)}{p} e^{(2p-3)\theta_n}\int_{\R^3}|v_n|^pdx
 -\frac{3}{2}e^{9\theta_n}\int_{\R^3}|v_n|^{6}dx\Bigg]h\\[1mm]
 &=\big\langle J'(\Phi(\theta_n,v_n)),\Phi(\theta_n,z)\big\rangle
        +G(\Phi(\theta_n,v_n))h.
\end{split}
\end{equation}
Taking $(h,z)=(1,0)$, a consequence of \eqref{eqS3.7}, we get
$$
G(u_n)\rightarrow0\quad\text{as}\quad n\rightarrow\infty.
$$
For any $v\in E$, set $z(x)=e^{-2\theta_n}v(e^{-\theta_n}x),\ h=0$ in \eqref{eqS3.9}, we conclude
$$
\big\langle J'(u_n),v\big\rangle=\big\langle J'(\Phi(\theta_n,v_n)),\Phi(\theta_n,z)\big\rangle=
o_n(1)\|\Phi(\theta_n,z)\|=o_n(1)\|v\|
$$
for $\theta_n\rightarrow0$ as $n\rightarrow\infty$. Thus, we have
$$
J'(u_n)\rightarrow0\ \text{in}\  E^{-1} \quad\text{as} \quad n\rightarrow\infty.
$$
Therefore, the lemma is proved.
\end{proof}

\begin{lemma}\label{lm3.5}
The sequence $\{u_n\}\subset E$ given in Proposition \ref{pro3.4} is bounded.
\end{lemma}

\begin{proof}
To show the boundedness of $\{u_n\}$, we note that
\begin{equation*}
\begin{split}
\ds  c+o_n(1)&=J(u_n)-\frac{1}{2p-3}G(u_n)\\[1mm]
        &=\frac{p-3}{2p-3} \int_{\R^3}|\nabla u_n|^2dx+\frac{q^2(p-3)}{2(2p-3)}\int_{\R^3}\phi_{u_n} u_n^2dx\\[1mm]
        &\quad+\frac{q^2}{4a(2p-3)}\int_{\R^3} \int_{\R^3} e^{-\frac{|x-y|}{a}}u_n^2(x)u_n^2(y)dxdy
                 +\frac{6-p}{3(2p-3)}\int_{\R^3}|u_n|^6dx,
\end{split}
\end{equation*}
which implies that $\|u_n\|$ is bounded in $E$.
\end{proof}

In order to get the existence of ground state solutions, the following lemma will be used.

\begin{lemma}\label{lm3.6}
If $u\in\mathcal{M}$, then $J(u)\ge c$.
\end{lemma}

\begin{proof}
For any $u\in\mathcal{M}$, it follows from Lemma \ref{lm3.2} that
\begin{equation}\label{eqS3.10}
J(u)=\max\limits_{t>0}J(u_t).
\end{equation}
Note that  $\lim\limits_{t\rightarrow+\infty}J(u_t)=-\infty$, then we can take $t_0$
large enough such that $J(u_{t_0})<0$. Define $\eta_0:[0,1]\rightarrow E$ by
\begin{equation*}
\renewcommand{\arraystretch}{1.25}
\begin{array}{ll}
\ds \eta_0(t)=\left \{ \begin{array}{ll}
\ds u_{t_0t}, &\quad t\in(0,1],\\
0,  &\quad t=0,\\
\end{array}
\right .
\end{array}
\end{equation*}
then we have $\eta_0\in\Gamma$. Combining this with \eqref{eqS3.10} we conclude that
$$
J(u)=\max\limits_{t>0}J(u_t)
           \geq\max\limits_{t\in[0,1]}J(\eta_0(t))\ge c.
$$
\end{proof}

In the following, we recall an important energy estimate for  $c$ which has been proved in \cite{LPT}.

\begin{lemma}\label{lm3.7}
If $p\in(4,6)$, then $c<\frac13 S^{\frac32}$ for all $\mu>0$; if $p\in(3,4]$, then $c<\frac13 S^{\frac32}$ for $\mu>0$ sufficiently large, where $S$ is the best Sobolev constant for the embedding $D^{1,2}(\R^3)\hookrightarrow L^{6}(\R^3)$.
\end{lemma}

\begin{lemma}\label{lm3.8}
There exist a sequence $\{y_n\}\subset\R^3$ and constants  $R,\ \beta>0$ such that for $q\in[3,6)$,
$$
\liminf\limits_{n\rightarrow\infty}\int_{B_R(y_n)} |u_n|^qdx\geq\beta>0,
$$
where the  sequence $\{u_n\}\subset E$ is given in Proposition \ref{pro3.4}.
\end{lemma}

\begin{proof}
Suppose by contradiction that for all $R>0$, we have
$$
\lim\limits_{n\rightarrow\infty}\sup\limits_{y\in\R^3}\int_{B_R(y)}|u_n|^qdx=0.
$$
By Lemma \ref{lm2.4}, we have that
\begin{equation*}
u_n\rightarrow0\quad\text{in}\quad L^\alpha(\R^3),\ \alpha\in(3,6).
\end{equation*}
Since $\big\langle J'(u_n),u_n\big\rangle=o_n(1)$, it follows   that
\begin{equation}\label{eqS3.11}
\ds   \int_{\R^3}|\nabla u_n|^2dx+q^2\int_{\R^3}\phi_{u_n} u_n^2dx-\int_{\R^3}|u_n|^6dx=o_n(1).
\end{equation}
By $J(u_n)\rightarrow c$, one has
\begin{equation}\label{eqS3.12}
\ds   \frac12\int_{\R^3}|\nabla u_n|^2dx+\frac{q^2}4\int_{\R^3}\phi_{u_n} u_n^2dx-\frac16\int_{\R^3}|u_n|^6dx=c+o_n(1).
\end{equation}
Recalling that $\{u_n\}\subset E$ is bounded, up to a subsequence, we may assume that
$$
\int_{\R^3}|\nabla u_n|^2dx\rightarrow l_1\geq0
$$
and
$$
\int_{\R^3}|u_n|^6dx\rightarrow l_2\geq0.
$$
It is easy to check that $l_1>0$. Otherwise, $l_2=0$ and $M(u_n)\rightarrow0$ as $n\rightarrow\infty$, which contradicts  $c>0$.
Combing this with \eqref{eqS3.11}, we also conclude that $l_2>0$.
From the definition of $S$, we have that
$$
\int_{\R^3}|\nabla u_n|^2dx\geq S\left(\int_{\R^3}|u_n|^6dx\right)^{\frac{1}{3}}.
$$
Taking the limit as $n\rightarrow\infty$, we obtain
\begin{equation}\label{eqS3.121}
\ds   l_1\ge Sl_2^{\frac13}\ge Sl_1^{\frac13},
\end{equation}
it follows that $l_1\ge S^{\frac32}$. On the other hand, from \eqref{eqS3.11}, \eqref{eqS3.12} and Lemma \ref{lm3.7}, we get
\begin{equation*}
\begin{split}
\ds \frac13 S^{\frac32}>c&=J(u_n)+o_n(1)\\[1mm]
        &=J(u_n)-\frac16\big\langle J'(u_n),u_n\big\rangle+o_n(1)\\[1mm]
    &=\frac13\int_{\R^3}|\nabla u_n|^2dx+\frac{q^2}{12}\int_{\R^3}\phi_{u_n} u_n^2dx+o_n(1)\\[1mm]
    &\ge\frac13 l_1,
\end{split}
\end{equation*}
which is a contradiction with $l_1\ge S^{\frac32}$.
\end{proof}

Now we are ready to prove Theorem \ref{thm1.1}.

\begin{proof}[\bf Proof of Theorem \ref{thm1.1}]
From Proposition \ref{pro3.4} and Lemma \ref{lm3.5}, we deduce a bounded $\text{(PS)}_{c}$ sequence $\{u_n\}$ for $J$ with
$G(u_n)\rightarrow0$.
 Denote $\tilde{u}_n(x)=u_n(x+y_n)$, where $\{y_n\}$ is the sequence given in Lemma \ref{lm3.8}.
It is easy to check that $\{\tilde{u}_n\}$ is still a bounded $\text{(PS)}_{c}$ sequence for $J$ with
$G(\tilde{u}_n)\rightarrow0.$
Up to a subsequence, we may assume that there is a $\tilde{u}\in E$ such that
\begin{equation}\label{eqS3.13}
\renewcommand{\arraystretch}{1.25}
\begin{array}{ll}
\ds \left \{ \begin{array}{ll}
\ds\tilde{u}_n\rightharpoonup\tilde{u}\quad\text{in}\quad E,\\
\tilde{u}_n\rightarrow\tilde{u}\quad\text{in}\quad L_{loc}^r(\R^3),\ r\in[1,6),\\
\tilde{u}_n\rightarrow\tilde{u}\quad\text{a.e. in}\quad \R^3.\\
\end{array}
\right .
\end{array}
\end{equation}

\noindent
It follows from Lemma \ref{lm3.8} that there exist  $R,\ \beta>0$ such that for $q\in[3,6)$,
$$
\int_{B_R(0)} |\tilde{u}_n|^qdx\geq\beta>0,
$$
which implies that $\tilde{u}\neq0$ due to \eqref{eqS3.13}. Moreover, since $J'(\tilde{u}_n)\rightarrow0$, we conclude that
\begin{equation}\label{eqS3.14}
\int_{\R^3}\nabla \tilde{u}_n\nabla\varphi dx+q^2\int_{\R^3}\phi_{\tilde{u}_n}\tilde{u}_n\varphi dx-\mu\int_{\R^3}|\tilde{u}_n|^{p-2}\tilde{u}_n\varphi dx
-\int_{\R^3}|\tilde{u}_n|^{4}\tilde{u}_n\varphi dx\rightarrow0
\end{equation}
for all $\varphi\in C_0^\infty(\R^3)$.
By virtue of the definition of weak convergence, we know that
$$
\int_{\R^3}\nabla \tilde{u}_n\nabla\varphi dx\rightarrow\int_{\R^3}\nabla \tilde{u}\nabla\varphi dx,
$$
$$
\int_{\R^3}|\tilde{u}_n|^{p-2}\tilde{u}_n\varphi dx\rightarrow\int_{\R^3}|\tilde{u}|^{p-2}\tilde{u}\varphi dx
$$
and
$$
\int_{\R^3}|\tilde{u}_n|^{4}\tilde{u}_n\varphi dx\rightarrow\int_{\R^3}|\tilde{u}|^{4}\tilde{u}\varphi dx.
$$
Furthermore, by Lemma \ref{lm2.3}, we also have that
\begin{equation*}
\begin{split}
\ds &\quad\left|\int_{\R^3}\phi_{\tilde{u}_n}\tilde{u}_n\varphi dx-\int_{\R^3}\phi_{\tilde{u}}\tilde{u}\varphi dx\right|\\[1mm]
        &\le \int_{\R^3}|\phi_{\tilde{u}_n}-\phi_{\tilde{u}}||\tilde{u}_n||\varphi|dx+\int_{\R^3}|\phi_{\tilde{u}}||\tilde{u}_n-\tilde{u}||\varphi|dx\\[1mm]
    &\le\left(\int_{supp(\varphi)}|\phi_{\tilde{u}_n}-\phi_{\tilde{u}}|^2dx\right)^{\frac12}\|\tilde{u}_n\|_6\|\varphi\|_3+\|\phi_{\tilde{u}}\|_6
    \left(\int_{supp(\varphi)}|\tilde{u}_n-\tilde{u}|^2dx\right)^{\frac12}\|\varphi\|_3\\[1mm]
    &\rightarrow 0\quad\text{as}\quad n\rightarrow\infty.
\end{split}
\end{equation*}
Using this and \eqref{eqS3.14}, we have
\begin{equation}\label{eqS3.15}
\int_{\R^3}\nabla \tilde{u}\nabla\varphi dx+q^2\int_{\R^3}\phi_{\tilde{u}}\tilde{u}\varphi dx-\mu\int_{\R^3}|\tilde{u}|^{p-2}\tilde{u}\varphi dx
-\int_{\R^3}|\tilde{u}|^{4}\tilde{u}\varphi dx=0
\end{equation}
for all $\varphi\in C_0^\infty(\R^3)$. By Lemma \ref{lm2.1}, \eqref{eqS3.15} holds also for all $\varphi\in E$ and so $\tilde{u}$ is a nontrivial critical
point of $J$. By Remark \ref{re3.1}, $\tilde{u}\in\mathcal{M}$. Then it follows from Lemma \ref{lm3.6} and Fatou's lemma that
\begin{align*}
\ds  c\leq J(\tilde{u})&=J(\tilde{u})-\frac13G(\tilde{u})\\[1mm]
                 &=\frac{q^2}{12a}\int_{\R^3} \int_{\R^3} e^{-\frac{|x-y|}{a}}\tilde{u}^2(x)\tilde{u}^2(y)dxdy
                 +\frac{\mu(2p-6)}{3p}\int_{\R^3}|\tilde{u}|^pdx+\frac{1}{3}\int_{\R^3}|\tilde{u}|^6dx\\[1mm]
                 &\leq\liminf\limits_{n\rightarrow\infty}\left[\frac{q^2}{12a}\int_{\R^3} \int_{\R^3} e^{-\frac{|x-y|}{a}}\tilde{u}_n^2(x)\tilde{u}_n^2(y)dxdy
                 +\frac{\mu(2p-6)}{3p}\int_{\R^3}|\tilde{u}_n|^pdx+\frac{1}{3}\int_{\R^3}|\tilde{u}_n|^6dx\right]\\[1mm]
                 &=\liminf\limits_{n\rightarrow\infty}\left[J(\tilde{u}_n)-\frac13G(\tilde{u}_n)\right]=c.
\end{align*}
Hence, $J(\tilde{u})=c$.

Finally, we end this proof by showing the existence of  a ground state solution for problem \eqref{eqS1.7}.
Set
$$
m=\inf\big\{J(u)\ \big|\ u\in E\setminus\{0\},\ J'(u)=0\big\}.
$$
We first claim that $m>0$. Let $u$ be such that $J'(u)=0$ and $J(u)=m$, we derive that
\begin{equation}\label{eqS3.16}
\begin{split}
\ds m&=J(u)-\frac{1}{2p-3}G(u)\\[1mm]
        &=\frac{p-3}{2p-3} \int_{\R^3}|\nabla u|^2dx+\frac{q^2(p-3)}{2(2p-3)}\int_{\R^3}\phi_{u} u^2dx\\[1mm]
        &\quad+\frac{q^2}{4a(2p-3)}\int_{\R^3} \int_{\R^3} e^{-\frac{|x-y|}{a}}u^2(x)u^2(y)dxdy
                 +\frac{6-p}{3(2p-3)}\int_{\R^3}|u|^6dx\\[1mm]
                 &\ge\frac{p-3}{2p-3}\min\left\{1,\frac{q^2}{2}\right\}M(u).
\end{split}
\end{equation}
Applying the inequality
$$
1-e^{-t}-te^{-t}\ge0,\quad\forall t\ge0,
$$
we get from \eqref{eqS3.2} that
\begin{align*}
\ds \frac{\mu(2p-3)}{p}  \int_{\R^3}|u|^pdx+\frac{3}{2}\int_{\R^3}|u|^{6}dx&=
\frac{3}{2}\int_{\R^3}|\nabla u|^2dx+\frac{q^2}{2}\int_{\R^3}\phi_uu^2dx\\[1mm]
&\quad+\frac{3q^2}{4}\int_{\R^3} \int_{\R^3} \frac{1-e^{-\frac{|x-y|}{a}}-|x-y|e^{-\frac{|x-y|}{a}}}{3|x-y|}u^2(x)u^2(y)dxdy\\[1mm]
  &\ge\min\left\{\frac32,\frac{q^2}{2}\right\}M(u).
\end{align*}
Combining this with Lemma \ref{lm2.6}, we have
$$
M(u)\le C\left(\|u\|_p^p+\|u\|_6^6\right)\le C\left[M(u)^{\frac{2p-3}{3}}+M(u)^3\right],
$$
that is $M(u)\geq C>0$. Hence, the claim follows by \eqref{eqS3.16}.

Let $\{u_n\}\subset E$ be a  sequence of nontrivial  critical points of $J$ satisfying $J(u_n)\rightarrow m\le J(\tilde{u})=c
<\frac13S^{\frac32}$.Using the earlier arguments, we can deduce that
$\{u_n\}$ is bounded in $E$. Arguing as before,
there exists a nontrivial  $u\in E$ such that $J(u)=m$ with $J'(u)=0$, that is,
$u$ is a ground state solution of problem \eqref{eqS1.7}.
\end{proof}

\begin{remarkn}\label{re3.9}
The solutions we obtain are easily seen to be positive. Indeed, if we repeat all the arguments above for the functional
$$
J^+(u)=\frac12\int_{\R^3}|\nabla u|^2dx+\frac{q^2}{4}\int_{\R^3}\phi_u u^2dx-\frac{\mu}{p}\int_{\R^3}(u^+)^pdx-\frac{1}{6}\int_{\R^3}(u^+)^6dx,
$$
where $u^+=\max\{u,0\}$. By the maximum principle, we can conclude $u>0$.
\end{remarkn}

\begin{remarkn}\label{re3.10}
We can also restrict ourselves to the subspace of radial functions $E_r\subset E$ from the beginning. In this case the proof of the convergence of
bounded (PS) sequence is easier by compactness of the embedding $E_r\hookrightarrow L^p(\R^3)$ (see Lemma \ref{lm2.2}).
By using the similar arguments as the proof of Theorem \ref{thm1.1}, we can also conclude that problem  \eqref{eqS1.7} has a positive ground state solution in $E_r$.
\end{remarkn}

\section{\label{sec4}Multiplicity for $p\in(4,6)$}

To establish the existence of multiple solutions for  \eqref{eqS1.7} in the case $p\in(4,6)$, we will focus on finding multiple critical points of $J$.
For this purpose, we will make use of an abstract critical point theorem recently developed by Perera \cite{Pe}.

Let $X$ be Banach space.  For a
symmetric subset $A$ of $X\setminus\{0\}$, we denote by $i(A)$ the cohomological index of $A$,
which was introduced by Fadell and Rabinowitz \cite{FR}. If $A$ is homeomorphic to
the unit sphere $S^{m-1}$ in $\R^m$, then $i(A)=m$.

\begin{proposition}\label{pro4.1}(\cite{Pe})
Let $X$ be a Banach space, $\Phi:X\rightarrow\R$ be an even $C^1$-functional satisfying $\textrm{(PS)}_c$ condition
for $c\in(0,c^*)$, where $c^*$ is some positive constant. If $0$ is a strict local minimizer of $\Phi$ and there are $R>0$ and a compact symmetric set $A\subset\partial\mathcal{B}_R$, where $\mathcal{B}_R=\{u\in X\ \big|\ \|u\|<R\}$,
such that  $i(A)=m$,
\begin{equation}\label{eqS4.1}
\max\limits_{u\in A}\Phi(u)\le0,\quad\max\limits_{u\in B}\Phi(u)<c^*,
\end{equation}
where $B=\{tu\ \big|\ t\in[0,1],\ u\in A\}$, then $\Phi$ has $m$ pairs of nonzero critical points
with positive critical values.
\end{proposition}

From now on we restrict ourselves to the radial subspace $E_r$. Then we apply Proposition \ref{pro4.1} by setting $X=E_r$ and $\Phi=J$.
First we prove a local (PS) condition for the functional $J$.

\begin{lemma}\label{lm4.2}
For $p\in(4,6)$, the functional $J$ defined on $E_r$ satisfies $\textrm{(PS)}_c$ condition if
\begin{equation}\label{eqS4.2}
0<c<\frac13S^{\frac32}.
\end{equation}
\end{lemma}

\begin{proof}
Let $\{u_n\}\subset E_r$ be a $\textrm{(PS)}_c$ sequence, i.e.,
$$
J(u_n)=\frac12\int_{\R^3}|\nabla u_n|^2dx+\frac{q^2}{4}\int_{\R^3}\phi_{u_n} u_n^2dx-\frac{\mu}{p}\int_{\R^3}|u_n|^pdx-\frac{1}{6}\int_{\R^3}|u_n|^6dx\rightarrow c
$$
and
\begin{equation}\label{eqS4.3}
\begin{split}
\big\langle J'(u_n),\varphi\big\rangle&=\int_{\R^3}\nabla u_n\nabla\varphi dx+q^2\int_{\R^3}\phi_{u_n}u_n\varphi dx-\mu\int_{\R^3}|u_n|^{p-2}u_n\varphi dx
-\int_{\R^3}|u_n|^{4}u_n\varphi dx\\[1mm]
&=o_n(1)\|\varphi\|
\end{split}
\end{equation}
for all $\varphi\in E_r$. Since  $p\in(4,6)$, we have
\begin{equation*}
\begin{split}
\ds  c+1+o_n(1)\|u_n\|&\ge J(u_n)-\frac{1}{p}\big\langle J'(u_n),u_n\big\rangle\\[1mm]
        &=\left(\frac12-\frac1p\right) \int_{\R^3}|\nabla u_n|^2dx+\left(\frac14-\frac1p\right)q^2\int_{\R^3}\phi_{u_n} u_n^2dx+\left(\frac1p-\frac16\right)\int_{\R^3}|u_n|^6dx\\[1mm]
        &\ge\left(\frac12-\frac1p\right) \int_{\R^3}|\nabla u_n|^2dx+\left(\frac14-\frac1p\right)q^2\int_{\R^3}\phi_{u_n} u_n^2dx,
\end{split}
\end{equation*}
then it follows that $\{u_n\}$ is bounded in $E_r$.
Up to a subsequence, by Lemma \ref{lm2.2} we may assume that there is a $u\in E_r$ such that
\begin{equation}\label{eqS4.4}
\renewcommand{\arraystretch}{1.25}
\begin{array}{ll}
\ds \left \{ \begin{array}{ll}
\ds u_n\rightharpoonup u\quad\text{in}\quad E_r,\\
u_n\rightarrow u\quad\text{in}\quad L^q(\R^3),\ q\in(\frac{18}{7},6),\\
u_n\rightarrow u\quad\text{a.e. in}\quad \R^3.\\
\end{array}
\right .
\end{array}
\end{equation}
Similar to  \eqref{eqS3.14} and  \eqref{eqS3.15}, we conclude $u$ is a critical point  of $J$, that is $J'(u)=0$. From this, we infer that
\begin{equation}\label{eqS4.5}
\begin{split}
J(u)&= J(u)-\frac{1}{p}\big\langle J'(u),u\big\rangle\\[1mm]
        &=\left(\frac12-\frac1p\right) \int_{\R^3}|\nabla u|^2dx+\left(\frac14-\frac1p\right)q^2\int_{\R^3}\phi_{u} u^2dx+\left(\frac1p-\frac16\right)\int_{\R^3}|u|^6dx\\[1mm]
        &\ge0.
\end{split}
\end{equation}

Set $v_n=u_n-u$. It follows from Brezis-Lieb lemma that
$$
\int_{\R^3}|\nabla u_n|^2dx=\int_{\R^3}|\nabla u|^2dx+\int_{\R^3}|\nabla v_n|^2dx+o_n(1)
$$
and
$$
\int_{\R^3}|u_n|^6dx=\int_{\R^3}|u|^6dx+\int_{\R^3}|v_n|^6dx+o_n(1).
$$
Combining this with \eqref{eqS4.4} and $J(u_n)\rightarrow c$, we have
\begin{equation}\label{eqS4.6}
J(u)+\frac12\int_{\R^3}|\nabla v_n|^2dx+\frac{q^2}{4}\left[\int_{\R^3}\phi_{u_n} u_n^2dx-\int_{\R^3}\phi_{u} u^2dx\right]-\frac{1}{6}\int_{\R^3}|v_n|^6dx= c+o_n(1).
\end{equation}
Taking $\varphi=u_n$ in \eqref{eqS4.3}, from $J'(u)=0$ we get
\begin{equation}\label{eqS4.7}
\int_{\R^3}|\nabla v_n|^2dx+q^2\left[\int_{\R^3}\phi_{u_n} u_n^2dx-\int_{\R^3}\phi_{u} u^2dx\right]-\int_{\R^3}|v_n|^6dx= o_n(1).
\end{equation}
Then, \eqref{eqS4.6} and \eqref{eqS4.7} yield
\begin{equation}\label{eqS4.8}
J(u)+\frac14\int_{\R^3}|\nabla v_n|^2dx+\frac{1}{12}\int_{\R^3}|v_n|^6dx= c+o_n(1).
\end{equation}
Up to a subsequence, we may assume that
$$
\int_{\R^3}|\nabla v_n|^2dx\rightarrow l_1\geq0
$$
and
$$
\int_{\R^3}|v_n|^6dx\rightarrow l_2\geq0.
$$
By Fatou's lemma, together with \eqref{eqS4.7}, indicates that $l_1\le l_2$. Similar to \eqref{eqS3.121} we have
\begin{equation*}
\ds   l_1\ge Sl_2^{\frac13}\ge Sl_1^{\frac13}.
\end{equation*}
If $l_1>0$, we get  $l_1\ge S^{\frac32}$. On the other hand, thanks to  \eqref{eqS4.5} and \eqref{eqS4.8}, we get a contradiction with \eqref{eqS4.2}:
$$
c\ge \frac14l_1+\frac1{12}l_2\ge\frac13l_1\ge\frac13S^{\frac32}.
$$
So, $l_1=0$, which means $v_n\rightarrow0$ in $D^{1,2}(\R^3)$.

Finally, we will show that $v_n\rightarrow0$ in $E_r$. Arguing as in Proposition 4.1 in \cite{MMV}, we obtain similar inequality as follows:
\begin{equation}\label{eqS4.9}
\int_{\R^3}\phi_{u_n} u_n^2dx-\int_{\R^3}\phi_{u} u^2dx\ge\int_{\R^3}\phi_{v_n} v_n^2dx+o_n(1).
\end{equation}
Combining this with \eqref{eqS4.7} gives
\begin{equation*}
\begin{split}
\ds \int_{\R^3}|\nabla v_n|^2dx+q^2\int_{\R^3}\phi_{v_n} v_n^2dx&\le\int_{\R^3}|v_n|^6dx+o_n(1)\\[1mm]
        &\le S^{-3}\left(\int_{\R^3}|\nabla v_n|^2dx\right)^3+o_n(1)\rightarrow0.
\end{split}
\end{equation*}
Therefore, $v_n\rightarrow0$ in $E_r$, that is $u_n\rightarrow u$ in $E_r$.
\end{proof}

\begin{proof}[\bf Proof of Theorem \ref{thm1.2}]
Arguing as in the proof of Lemma \ref{lm3.3} (i), it is clear that $u=0$ is a strict local minimizer of $J$. By Lemma \ref{lm4.2}, $J$
satisfies $\textrm{(PS)}_c$ condition
for $c\in(0,c^*)$ with $c^*=\frac13S^{\frac32}$. Then it suffices to find the subsets $A$ and $B$ satisfying the geometric assumption
\eqref{eqS4.1} in applying Proposition \ref{pro4.1}. Borrowing the argument from \cite{LP}, for fixed $m\in\N$, let $Z_m$ be  an $m$-dimensional subspace of $E_r$.
For any $u\in Z_m$, we have
\begin{equation}\label{eqS4.10}
\begin{split}
J(u)&=\frac 12\int_{\R^3}|\nabla u|^2dx+\frac{q^2}{4}\int_{\R^3}\phi_u u^2dx-\frac{\mu}{p}\int_{\R^3}|u|^pdx-\frac{1}{6}\int_{\R^3}|u|^6dx\\[1mm]
        &\le \frac12\|u\|^2+\frac{q^2}{4}\|u\|^4-\mu C_1\|u\|^p-C_2\|u\|^6,
\end{split}
\end{equation}
where we used the fact that all norms on $Z_m$ are equivalent. Take $R>0$ large enough such that
\begin{equation}\label{eqS4.11}
f(R):=\frac12R^2+\frac{q^2}{4}R^4-C_2R^6<0.
\end{equation}
Let $A=Z_m\cap\partial\mathcal{B}_R$, then $i(A)=m$. If $\mu>0$, then for any $u\in A$, \eqref{eqS4.10} gives that $J(u)\le f(R)<0$. Thus
$$
\max\limits_{u\in A}J(u)\le0.
$$
In view of \eqref{eqS4.11}, we may take $\delta\in(0,R)$ small enough such that
\begin{equation}\label{eqS4.12}
f(s)<c^*\quad\text{for}\ s\in[0,\delta].
\end{equation}
Set
$$
\mu_m=1+\max\limits_{s\in[\delta,R]}\left|\frac{f(s)-c^*}{C_1s^p}\right|.
$$
If $\mu\ge\mu_m$ we have
\begin{equation}\label{eqS4.13}
f(s)-\mu C_1s^p<c^*\quad\text{for}\ s\in[\delta,R].
\end{equation}
Then we claim that for $B=\{tu\ \big|\ t\in[0,1],\ u\in A\}$, there holds
$$
\max\limits_{u\in B}J(u)<c^*.
$$
Now we consider two cases.

\noindent
(i)\ If $t\in[0,\frac{\delta}{R}]$, then $\|tu\|\in[0,\delta]$, and \eqref{eqS4.12} implies
$$
J(tu)\le f(\|tu\|)<c^*.
$$
(ii)\ If $t\in[\frac{\delta}{R},1]$, then $\|tu\|\in[\delta,R]$, and \eqref{eqS4.13} implies
$$
J(tu)\le f(\|tu\|)-\mu C_1\|tu\|^p<c^*.
$$
Hence, the claim follows. By Proposition \ref{pro4.1}, $J$ has  $m$ pairs of nontrivial critical points
with positive critical values, which are  $m$ pairs of nontrivial solutions for problem \eqref{eqS1.7}.
\end{proof}

\medskip
\noindent
\textbf{Author contributions} Wentao Huang wrote the main manuscript text and Li Wang reviewed and revised
 the manuscript.

\medskip
\noindent
\textbf{Data Availability Statement} The manuscript has no associated data.

\medskip
\noindent
\textbf{Declarations}

\noindent
\textbf{Conflict of interest} The authors declare that there is no conflict of interest.

\medskip
\noindent
\textbf{Ethical Approval} The authors declare that there are no ethics issues to be approved or disclosed.


\begin{thebibliography}{99}
\bibitem{BL} H. Berestycki, P. L. Lions, Nonlinear scalar field equations. I. Existence of a ground state,
              {\it Arch. Rational Mech. Anal.,} {\bf 82} (1983), 313--345.

\bibitem{Bo} F. Bopp, Eine Lineare Theorie des Elektrons, {\it Ann. Phys.,} {\bf430} (1940), 345--384.

\bibitem{B1} M. Born, Modified field equations with a finite radius of the electron, \textit{Nature,} \textbf{132} (1933), 282.

\bibitem{B2} M. Born, On the quantum theory of the electromagnetic field, \textit{Proc. R. Soc. Lond. Ser. A,} \textbf{143} (1934), 410--437.

\bibitem{B3} M. Born, L. Infeld, Foundations of the new field theory, \textit{Nature,} \textbf{132} (1933), 1004.

\bibitem{B4} M. Born, L. Infeld, Foundations of the new field theory, \textit{Proc. R. Soc. Lond. Ser. A,}  \textbf{144} (1934), 425--451.

\bibitem{Cd}E. Caponio, P. d'Avenia, A. Pomponio,  G. Siciliano, L. Yang, On a nonlinear Schr\"{o}dinger-Bopp-Podolsky system in the zero mass case:
Functional framework and existence, {\it arXiv:2506.09752v1,} (2025).

\bibitem{CT} S. Chen, X. Tang, On the critical Schr\"{o}dinger-Bopp-Podolsky system with general nonlinearities,
 \textit{Nonlinear Anal.,} \textbf{195} (2020), 111734.

\bibitem{DS}H. M. S. Damian, G. Siciliano, Critical Schr\"{o}dinger-Bopp-Podolsky systems: solutions
 in the semiclassical limit, \textit{Calc. Var. Partial Differential Equations,} \textbf{63} (2024), 63:155.

\bibitem{dG}P. d'Avenia, M. G. Ghimenti, Multiple solutions and profile description for a nonlinear Schr\"{o}dinger-Bopp-Podolsky-Proca system on a manifold, \textit{Calc. Var. Partial Differential Equations,} \textbf{61} (2022), 61:223.

\bibitem{d} P. d'Avenia, G. Siciliano, Nonlinear Schr\"{o}dinger euqation in the Bopp-Podolsky electrodynamics: Solutions in the electrostatic case, {\it J. Differential Equations,} {\bf 267} (2019), 1025--1065.

\bibitem{de}G. de Paula Ramos, Concentrated solutions to the Schr\"{o}dinger-Bopp-Podolsky system with a positive potential, \textit{J. Math. Anal. Appl.,} \textbf{535} (2024),  128098.

\bibitem{dS}G. de Paula Ramos, G. Siciliano, Existence and limit behavior of least energy solutions to constrained Schr\"{o}dinger-Bopp-Podolsky systems in $\R^3$,
\textit{Z. Angew. Math. Phys.,} \textbf{74} (2023), 74:56.

\bibitem{FR}E. R. Fadell, P. H. Rabinowitz, Generalized cohomological index theories for
Lie group actions with an application to bifurcation questions for Hamiltonian
systems, \textit{Invent. Math., } \textbf{45} (1978), 139--174.


\bibitem{Fr}F. Frenkel, $4/3$ problem in classical electrodynamics, \textit{Phys. Rev. E,} \textbf{54} (1996), 5859--5862.


\bibitem{H1}E. Hebey, Electro-magneto-static study of the nonlinear Schr\"{o}dinger equation coupled with Bopp-Podolsky electrodynamics in the Proca setting, \textit{Discrete Contin. Dyn. Syst.,} \textbf{39} (2019), 6683--6712.

\bibitem{H2}E. Hebey, Strong convergence of the Bopp-Podolsky-Schr\"{o}dinger-Proca system to the Schr\"{o}dinger-Poisson-Proca system
in the electro-magneto-static case, \textit{Calc. Var. Partial Differential Equations,} \textbf{59} (2020), 59:198.

\bibitem{HIT}J. Hirata, N. Ikoma, K. Tanaka, Nonlinear scalar feld equations in $\R^N$: Mountain pass and symmetric mountain pass
approaches, \textit{Topol. Methods Nonlinear Anal.,} \textbf{35} (2010),  253--276.

\bibitem{HW}J. Huang, S. Wang, Normalized ground states for the mass supercritical Schr\"{o}dinger-Bopp-Podolsky system: Existence,
uniqueness, limit behavior, strong instability, {\it J. Differential Equations,} {\bf 437} (2025), 113282.

\bibitem{IR} I. Ianni, D. Ruiz, Ground and bound states for a static Schr\"{o}dinger-Poisson-Slater problem, \textit{Commun. Contemp. Math.,}
\textbf{14} (2012), 1250003.

\bibitem{J}L. Jeanjean, Existence of solutions with prescribed norm for semilinear elliptic equations, \textit{Nonlinear Anal.,} \textbf{28} (1997),
 1633--1659.

\bibitem{LPT} L. Li, P. Pucci, X. Tang, Ground state solutions for the nonlinear Schr\"{o}dinger-Bopp-Podolsky system
 with critical Sobolev exponent, \textit{Adv. Nonlinear Stud.,} \textbf{20} (2020), 511--538.

\bibitem{LZ}Y. Li, B. Zhang, Critical Schr\"{o}dinger-Bopp-Podolsky system with prescribed mass, \textit{J. Geom. Anal.,} \textbf{7} (2023), 1--27.

\bibitem{LCF} Y. Li, X. Chang, Z. Feng, Normalized solutions for Sobolev critical Schr\"{o}dinger-Bopp-Podolsky systems,
 \textit{Electron. J. Differential Equations,} \textbf{56} (2023), 1--19.

\bibitem{Lions}P. L. Lions, The concentration-compactness principle in the calculus of variations. The locally compact case. II, \textit{Ann.
Inst. H. Poincar\'{e} Anal. Non Lin\'{e}aire,} \textbf{2} (1984), 223--283.

\bibitem{LP}S. Liu, K. Perera, Multiple solutions for $(p,q)$-Laplacian equations in $\R^N$
with critical or subcritical exponents, \textit{Calc. Var. Partial Differential Equations,}
\textbf{63} (2024), 63:199.

\bibitem{LZH}Z. Liu, Z. Zhang, S. Huang, Existence and nonexistence of positive solutions for a static Schr\"{o}dinger-Poisson-Slater
equation, \textit{J. Differential Equations,} \textbf{266} (2019), 5912--5941.

\bibitem{MMV}C. Mercuri, V. Moroz, J. Van Schaftingen, Groundstates and radial solutions to nonlinear Schr\"{o}dinger-Poisson-Slater
equations at the critical frequency, \textit{Calc. Var. Partial Differential Equations,} \textbf{55} (2016), 55:146.

\bibitem{Mi} G. Mie, Grundlagen einer Theorie der Materie, {\it Ann. Phys.,} \textbf{345} (1913), 1--66.

\bibitem{Pe}K. Perera, Abstract multiplicity theorems and applications to critical growth
problems, \textit{J. Anal. Math.,} (in press).

\bibitem{Po} B. Podolsky, A generalized electrodynamics, {\it Phys. Rev.,} {\bf 62} (1942), 68--71.

\bibitem{R}D. Ruiz, On the Schr\"{o}dinger-Poisson-Slater system: behavior of minimizers, radial and nonradial cases, \textit{Arch. Ration.
Mech. Anal.,} \textbf{198} (2010), 349--368.

\bibitem{SS} G. Siciliano, K. Silva, The fibering method approach for a non-linear Schr\"{o}dinger equation coupled with the electromagnetic field, \textit{Publ. Mat.,} \textbf{64} (2020), 373--390.

\bibitem{WCL}L. Wang, H.  Chen, S.  Liu, Existence and multiplicity of sign-changing solutions for a Schr\"{o}dinger-Bopp-Podolsky system, \textit{Topol. Methods Nonlinear Anal.,} \textbf{59} (2022), 913--940.

\bibitem{W} M. Willem, Minimax Theorems, Progr. Nonlinear Differential Equations Appl., vol. 24, Birkh\"{a}user Boston, Inc.,
Boston, MA, 1996.

\bibitem{Z1}Q.  Zhang, Sign-changing solutions for Schr\"{o}dinger-Bopp-Podolsky system with general nonlinearity, \textit{Z. Angew. Math. Phys.,} \textbf{73} (2022), 73:235.

\bibitem{Z2}Z.  Zhang, Sign-changing solutions for a class of Schr\"{o}dinger-Bopp-Podolsky system with concave-convex nonlinearities, \textit{J. Math. Anal. Appl.,} \textbf{530} (2024),  127712.
\end{thebibliography}
\end{document}